\documentclass[12pt]{article}

\usepackage{graphicx}
\usepackage{amsfonts}
\usepackage{amsmath}
\usepackage{amsthm}
\usepackage{amssymb}
\usepackage{enumitem}
\usepackage{mathrsfs}
\usepackage{multirow}
\usepackage{subcaption}
\usepackage{hyperref}
\hypersetup{
	colorlinks=true,
	linktoc=all,
	linkcolor=blue,
}
\usepackage{geometry}

\graphicspath{{./Figures/}}

\newtheorem{theorem}{Theorem}[section]
\newtheorem{lemma}[theorem]{Lemma}
\newtheorem{corollary}[theorem]{Corollary}

\newtheorem{observation}[theorem]{Observation}
\newtheorem{definition}[theorem]{Definition}

\begin{document}

\title{The structure of graphs with no $W_4$ immersion}

\author{
   Matt DeVos\thanks{Department of Mathematics, Simon Fraser University, Burnaby, B.C., Canada V5A 1S6, mdevos@sfu.ca.
     Supported in part by an NSERC Discovery Grant (Canada).}
\and
   Mahdieh Malekian\thanks{Department of Mathematics, Simon Fraser University, Burnaby, B.C., Canada V5A 1S6, mahdieh\_malekian@sfu.ca}
}

\date{}

\maketitle

\begin{abstract}
This paper gives a precise structure theorem for the class of graphs which do not contain $W_4$ as an immersion.  This strengthens a previous result of Belmonte at al. \cite{w4freeBelmonte} that gives a rough description of this class.  In fact, we prove a stronger theorem concerning rooted immersions of $W_4$ where one terminal is specified in advance.  This stronger result is key in a forthcoming structure theorem for graphs with no $K_{3,3}$ immersion.
\end{abstract}

%
%

\section{Introduction} 
\label{sec-intro}
Throughout this paper, we will consider finite loopless undirected graphs which are permitted to have parallel edges.  Let $G$ be a graph and let $xy,yz \in E(G)$ be a pair of distinct edges (with a common neighbour $y$). Then we say $xy, yz$ is  \emph{split off at} $y$ if we delete these edges and add a new edge $xz$. If a graph isomorphic to $H$ can be obtained from a subgraph of $G$ by a sequence of splittings, then we say that $H$ is \emph{immersed} in $G$ or equivalently we say that $G$ has an $H$ \emph{immersion}, and write $G \succ H$ or $H\prec G$.  

There is an alternate way of describing an immersion which is sometimes more useful:   Equivalently a graph $G$ has an immersion of $H$ if there is a function $\phi$ which maps the vertices of $H$ injectively to the vertices of $G$, maps every edge $uv$ of $H$ to a path from $\phi(u)$ to $\phi(v)$ in $G$, and has the property that the paths $\phi(e)$ and $\phi(f)$ are edge-disjoint whenever $e,f \in E(H)$ are distinct.  The vertices in $\phi( V(H) )$ are called the \emph{terminals} of this immersion.  Let us note that there is a stronger form of immersion, known as \emph{strong immersion}, which has the added constraint that the paths $\phi(e)$ must be internally disjoint from the terminals.  We will focus on the weak version which we continue to refer to as immersion.

In the setting of graph minors there are a great number of theorems which give a precise structural description of the class of graphs with no $H$-minor for various small graphs $H$.  Notably, such characterizations exist when $H$ is $K_{3,3}$, $K_5$, $W_4$, $W_5$, Prism, Octahedron, Cube, and $V_8$ \cite{MR1513158, MR0140094, MR999699, MR0157370, MR3090713, MR1794690, MR3548299}.  Here $W_k$ denotes a graph obtained from a cycle of length $k$ by adding one new vertex adjacent to every other, and $V_8$ is Wagner's graph.  In sharp contrast to this, there are very few graphs $H$ for which we have a good description of the class of graphs which do not contain an immersion of $H$.  The most significant works on this to date are as follows.
\begin{enumerate}[label=(\arabic*)]
\item Booth et al. give a handful of structure theorems regarding graphs with no $K_4$ immersion and use them to obtain a fast algorithm for testing the existence of a $K_4$ immersion \cite{MR1671844}.  
\item Belmonte et al. give a rough description of the class of graphs with no $W_4$ immersion  \cite{w4freeBelmonte}.
\item Giannopoulou et al. give a rough description of the class of graphs with no $K_{3,3}$ or $K_5$ immersion \cite{MR3283573}.
\end{enumerate}

In this article we will give a precise structure theorem for the class of graphs with no $H$ immersion when $H=K_4$ and when $H=W_4$ thus improving upon (1) and (2).  In fact, we will prove even stronger variants of these results concerning rooted immersions of $K_4$ and $W_4$.  The latter of these is a key ingredient in a forthcoming paper which improves upon (3) by giving a precise structure theorem for graphs with no $K_{3,3}$ immersion.  

We require a bit of notation before we are ready to state the theorems of interest.  For a graph $G = (V, E)$ and $X\subseteq V$, we denote by $\delta_G(X)$ the set of edges which have exactly one endpoint in $X$, and we let $d_G(X)= |\delta_G(X)|$ (when the graph concerned is clear from the context, we may drop the subscript).  A graph $G$ is said to be \emph{$k$-edge-connected} (\emph{internally $k$-edge-connected}) if $d(X) \ge k$ whenever $|X|, |V\setminus X| \ge 1$ ($\ge 2$).  Now we may state the theorem of Belmonte et. al. (The definition of tree-width is postponed to Section \ref{sec-w4}).

\begin{theorem}[Belmonte, Giannopoulou, Lokshatanov, and Thilikos \cite{w4freeBelmonte}]
\label{belmonte-w4-intro}
Let $G$ be a $3$-edge-connected, and internally $4$-edge-connected graph. If $G$ does not immerse $W_4$, then either $G$ is cubic or the tree-width of $G$ is bounded by an absolute constant.
\end{theorem}

This theorem effectively divides those graphs with no $W_4$ immersion into cubic graphs (that obviously cannot immerse $W_4$) and a controlled class with bounded tree-width.  However, the absolute constant coming from their proof is quite large.  As a corollary of our precise structure theorem, we show that the best possible value for this constant is three.  

Although the above theorem is stated under the assumptions that the graph $G$ is 3-edge-connected and internally 4-edge-connected, this theorem can be used to describe the structure of an arbitrary graph $G$ with no $W_4$ immersion.  Before explaining this, let us pause to introduce a useful definition:  For a graph $H$ and $X \subset V(H)$, we denote by $H.X$ the graph obtained from $H$ by identifying $X$ to a single vertex (followed by deletion of any loops created in this process).  Now, suppose that $H$ is a 3-edge-connected and internally 4-edge-connected graph and we are interested in determining whether an arbitrary graph $G$ has an $H$ immersion.  We show how to use the edge-connectivity properties of $H$ to obtain some stronger assumptions on the edge-connectivity of $G$. Let $G'$ be the graph obtained from $G$ by deleting every cut-edge.  This graph $G'$ satisfies $H\prec G$ if and only if $H \prec G'$, but every component of $G'$ is 2-edge-connected.  Next, suppose there is a component $K$ of $G'$ with a set $X \subset V(K)$ such that $d(X) = 2$.  Form a new graph $K'$ ($K''$) from $K.X$ ($K.(V(K) \setminus X)$) by suppressing the newly created degree two vertex.  Then modify the graph $G'$ by removing the component $K$ and then adding the components $K'$ and $K''$.  If $G''$ is the graph resulting from repeating this modification wherever applicable, then $G''$  will satisfy $H \prec G$ if and only if $H \prec G''$, but now every component of $G''$ is 3-edge-connected.  Finally, suppose that $K$ is a component of $G''$ with a set $X \subset V(K)$ satisfying $d(X) = 3$ and $|X|, |V(K) \setminus X| \ge 2$.  In this case we modify $G''$ by deleting $K$ and adding the components $K.X$ and $K.(V(K) \setminus X)$.  If this operation is repeated until no longer possible, the resulting graph $G'''$ will still satisfy $H \prec G$ if and only if $H \prec G'''$, however every component of $G'''$ will be 3-edge-connected and internally 4-edge-connected.  Since it is generally easier to state our results under some connectivity assumptions, we will continue to do so.  However in all cases our structure theorems may be applied by this reasoning to arbitrary graphs.

To be able to state our theorem on the structure of graphs excluding $W_4$ immersion, we need to introduce a reduction that features in our result. We call a graph $H$ a \emph{doubled cycle} \emph{(doubled path)} if it can be obtained from a cycle (path) by adding a second copy of each edge.  

\begin{definition}
	\normalfont
	Let $G$ be a graph and let $X \subset V(G)$ satisfy $ |X| =k \ge 2$. We say $G[X]$ is a {\it chain of sausages of order} $k$ in $G$ if $G. (V(G)\setminus X )$ is a doubled cycle (of length $k+1$).  If $G[X]$ is a chain of sausages of order at least 3 and $x,x' \in X$ are adjacent, then the operation of identifying $x$ and $x'$ to a new vertex is called a \emph{sausage shortening}.  If $G'$ is obtained from $G$ by repeatedly performing sausage shortenings until this operation is no longer possible, we call $G'$ a \emph{sausage reduction} of $G$.  We say that $G$ is \emph{sausage reduced} if it has no sausage of order at least 3 (so no sausage shortening is possible).  
\end{definition} 

\begin{figure}[htbp]
	\centering
	\includegraphics{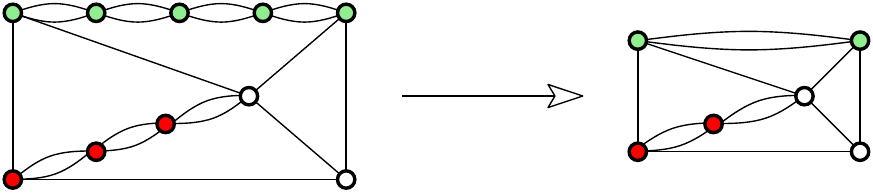}
	\caption{Sausage reduction}
	\label{fig:s-red}
\end{figure}
In our complete characterization of graphs excluding $W_4$ immersion, stated in Section \ref{sec-w4}, two sporadic (sausage reduced) graphs on five vertices appear. However, for sausage reduced graphs on at least six vertices our result asserts the following.

\begin{theorem}
	\label{w4-thm-intro}
	Let $G$ be a $3$-edge-connected, and internally $4$-edge-connected graph which is sausage reduced. If $|V(G)|\ge 6$, then $G\nsucc W_4$ if and only if $G$ is cubic.
\end{theorem}

There are two aspects to our proof that are especially valuable.  First, we make heavy use of structure theorems for rooted immersions of certain smaller graphs.  These theorems about rooted immersions allow us to test for the presence of a $W_4$ immersion that splits in a convenient way over a small edge cut.  Second, to handle all of the exceptions that appear on small numbers of vertices, we use a computer.  For our rooted $W_4$ theorem we did an exhaustive search of all graphs on up to 8 vertices.  This allows us to work at a level above all of the sporadic exceptions in our proof, thus greatly simplifying the argument.

The rest of this paper is organized as follows: In Section \ref{sec-dm}, we will see a result which characterizes the graphs which do not contain certain rooted graphs on four vertices as immersion.  As a corollary we obtain structure theorems for rooted immersions of $K_4$ with 0, 1, and 2 roots.  
 Section \ref{sec-rw4} is devoted to a structural theorem for graphs which exclude rooted $W_4$ as immersion, and finally in Section \ref{sec-w4}, we obtain a structural theorem for graphs which do not immerse $W_4$.

\section{Forbidding a $D_m$ immersion}
\label{sec-dm}

In this section we will prove a structure theorem for graphs not containing a certain rooted graph on four vertices called $D_m$.  This will be a key tool in establishing the structure of graphs with no $W_4$ immersion.  

%
\subsection{Statement of the main theorem}
Let us start by precisely defining a rooted graph, and the corresponding notion of immersion for rooted graphs. A {\emph {rooted graph}} is a connected graph $G$ equipped with an ordered tuple $(x_1,  \ldots , x_k)$ of distinct vertices. If $H$ together with $(y_1, \ldots, y_k)$ is another rooted graph, we say $G$ contains $H$ as a {\it rooted immersion} if there is a sequence of splits and deletions which transforms $G$ into a graph isomorphic to $H$, where this isomorphism sends $x_i$ to $y_i$, for $i=1,\ldots, k$, denoted $(G; x_1, \ldots ,x_k) \succ_r (H; y_1, \ldots , y_k)$. For the sake of simplicity, if $k= 2$ and there is an automorphism of $H$ which sends $y_1$ to $y_2$, we simply refer to $H$ as a rooted graph with roots $y_1, y_2$. For clarity, in our figures the roots are always solid while other vertices are open. 

The family of the rooted graphs concerned in this section is introduced below. For $u, v$ distinct vertices of a graph $G$, we let $E_G(u, v)$, $e_G(u,v)$ denote the set of edges between $u, v$, and the size of this set, respectively. 
For $m\ge 2$, let $D_m$ denote the graph with roots $x_0, x_1$ where $e(x_0, x_1) = m-2$, and $D_m\setminus E(x_0, x_1)$ is isomorphic to the rooted graph below.
\begin{figure}[htbp]
\centering
\includegraphics{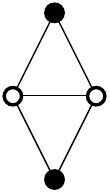}
\caption{Graph $D_2$}
\end{figure}

Observe that $D_3$ is isomorphic with $K_4$ with two roots. Accordingly, our result on excluding a $D_m$ immersion can be used to characterize (unrooted) graphs without $K_4$ immersion. This is done in Section \ref{sec-k4}. Next, we will introduce obstructions to immersion of $D_m$. To do so, we start by describing a very particular structure consisting of small nested edge-cuts.

\begin{definition}
	\normalfont
	If $G$ is a graph and $X,Y \subseteq V(G)$ are disjoint, a \emph{segmentation} of $G$ \emph{relative to} $(X,Y)$ is a family of nested sets $X_1 \subset X_2 \subset X_3 \ldots \subset X_k$ satisfying:
	\begin{itemize}
		\item $X_1 = X$ and $V(G) \setminus X_k = Y$
		\item $|X_{i+1} \setminus X_i| = 1$ for $1 \le i \le k-1$
	\end{itemize}
We say that the segmentation has \emph{width} $k$ if $d(X_i) = k$ holds for every $1 \le i \le k$ and we call it an $(a,b)$-\emph{segmentation} if $|X| \le a$ and $|Y| \le b$.
\begin{figure}[htbp]
	\centering
	\includegraphics{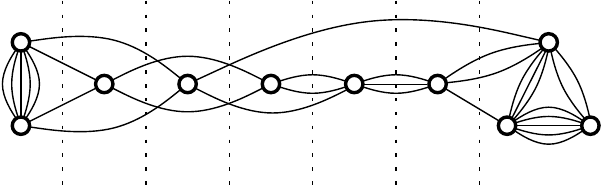}
	\caption{A graph with a $(2, 3)$-segmentation of width four}
	\label{fig:segmentation-ex}
\end{figure}
We refer to $X$ as {\it head}, and to $Y$ as {\it tail} of the segmentation.
\end{definition}

We now introduce two families of graphs which do not immerse $D_m$. Let $G$ be a graph with two roots $x_0, x_1$. If $C$ is a component of $G\setminus \{x_0, x_1\}$, we call $G[C \cup \{x_0, x_1\}] \setminus E(x_0, x_1)$ a \emph{lobe} of $G$. We then say $G$ has
\begin{description}
	\item[Type $A_m$.] if it has a segmentation of width $m$ relative to some $(X_0, X_1)$ with $|X_i|\le 2$ and $x_i \in X_i$, for  $i=0, 1$.
	\item[Type $B_m$.] if $G$ satisfies the following:
	\begin{itemize}
		\item
		Every lobe $L$ of $G$ consists of an $x_0-x_1$-path, in which there is exactly $n_L$ copies of each edge, for some $n_L \ge 2$.
		\item
		If $|V(L)|\ge 4$, we have $n_L=2$ (so $L$ is a doubled $x_0-x_1$ path)
		\item
		$e_G(x_0, x_1) + \sum_{L} n_L = m+1$, where the sum is taken over all lobes $L$ of $G$.
	\end{itemize}
\end{description}

To be able to state our main result, we need one more ingredient to capture a more refined notion of edge-connectivity for graphs of our interest in this section (which have two roots).
\begin{definition}
	\normalfont
	Let $G$ be a connected graph with roots $x_0, x_1$.
	\begin{itemize}
		\item $\lambda_s (G)$ ($\lambda^i_s (G)$) denotes
		the minimum size of an (internal) edge-cut separating $x_0$ and $x_1$.
		\item
		$\lambda_n (G)$ ($\lambda^i_n (G)$) denotes the minimum size of an (internal) edge-cut not separating $x_0$ and $x_1$. 
	\end{itemize}
\end{definition}

 We are now all set to state our main theorem concerning the structure of graphs excluding $D_m$ immersion, for any $m\ge 2$.

\begin{theorem}
	\label{Dm-thm}
	Let $m \ge 2$, and let $G$ be a rooted graph with roots $x_0, x_1$, where $|V(G)| \ge 4$.  Assume further that $\lambda_n (G)\ge 3$, $\lambda_n^i (G)\ge 4$, and $\lambda_s(G) \ge m$. 
	Then $G\nsucc_r D_m$ if and only if $G$ has type $A_m$ or type $B_m$.
\end{theorem}

\subsection{Proof of the `if' direction}
\label{sec-pf-Dm-thm}
In this subsection, we will see the proof of the easier direction, `if' direction of Theorem \ref{Dm-thm}. The following simple observation is a helpful tool in doing so.
\begin{observation}
	Let $G$ be a graph with roots $x_0, x_1$ with $G \succ_r D_m$.
	\begin{enumerate}
		\item 
		\label{seg-m}
		Suppose $G$ has a segmentation $X_0\subset X_1\subset \ldots \subset X_k$ of width $m$ so that $x_0 \in X_0$ and $x_1 \notin X_k$. If $|X_0|\le 2$, then $T\cap X_k = \{x_0\}$.
		\item 
		\label{parity-big-target}
		If $d(x_0)$ and $m$ have different parity, then an edge incident with $x_0$ can be deleted while preserving an immersion of $D_m$.
	\end{enumerate}
\end{observation}
\begin{proof}
	The fist part follows from $\lambda_s^i (D_m) \ge m+1$, and the second part is an immediate consequence of our definitions.
\end{proof}
\begin{proof} [Proof of the `if' direction of Theorem \ref{Dm-thm}]
	It follows immediately from part \ref{seg-m} of the previous observation that a graph of type $A_m$ does not immerse $D_m$. Suppose $G$ is type $B_m$. If $m=2$, observe that $e_G(x_0, x_1) =1$ and $G\setminus E(x_0, x_1)$ is a doubled $x_0-x_1$-path. Then part \ref{parity-big-target} of the above observation implies that $G \succ_r D_2$ iff a graph $G'$ obtained from $G$ by deleting an edge incident with $x_0$ immerses $D_2$. 
	However, every such a graph $G'$ either has type $A_2$ or $\lambda_n^i (G') <4$, so $D_2$ is not immersed in $G'$.
	
	Now consider $m\ge 3$. If $e(x_0, x_1)>0$, let $G'=G- x_0x_1$, and observe that $G\succ_r D_m$ iff $G'\succ_r D_{m-1}$. However, since $G'$ is type $B_{m-1}$, by induction, we have $G'\nsucc_r D_{m-1}$. If $e(x_0, x_1) =0$, then $G$ immerses $D_m$ iff the graph $G'$ resulting by splitting off an $x_0x_1$-path in a lobe $L$ of $G$ immerses $D_m$. However, then $G'$ has type $B_m$ with an edge between $x_0, x_1$, so $D_m\nprec G'$.
\end{proof}
\subsection{Graphs on four vertices}
As the reader may expect, the proof of the reverse direction is more involved. We first prove the theorem for four-vertex graphs.
\begin{lemma}
	\label{atleast4}
	Theorem \ref{Dm-thm} holds under the added assumption $|V(G)| = 4$.
\end{lemma}

\begin{proof}
	Assume that $V(G) = \{x_0,x_1,y_0,y_1\}$ with roots $x_0$ and $x_1$.  We may assume that $G$ satisfies $\lambda_s^i(G) \ge m+1$, as otherwise $G$ is type $A_m$ relative to $\{x_0, y_i\}$, $\{x_1, y_{1-i} \}$ for some $i=0,1$.  Let $G^*$ be the graph obtained from the simple graph underlying $G$ by deleting the edge $x_0 x_1$ if it is present.  First suppose that $|E(G^*)| = 5$.  In this case $G$ has a subgraph $H$ isomorphic to $D_2$ in which $d_H(x_i) = 2$ for $i=0,1$.  It follows from $\lambda_s(G) \ge m$ and $\lambda_s^i(G) \ge m+1$ that the graph $G' = G \setminus E(H)$ satisfies $\lambda_s(G') \ge m-2$, and thus $G$ has an immersion of $D_m$.
	
	Next suppose that $d_{G^*}(y_0) = d_{G^*}(y_1) = 1$.  In this case $G$ has one of the first two graphs (from the left) in Figure \ref{figy1x0} as a subgraph, so $D_m \prec G$.  Next suppose that $d_{G^*}(y_0) = 1$ and $d_{G^*}(y_1) \ge 2$.  If $y_0$ is adjacent to a root, say $x_0$, then $G$ immerses the middle graph in Figure \ref{figy1x0}; if $y_0$ is adjacent to $y_1$, then $G$ immerses the second graph from the right in Figure \ref{figy1x0}.  Since both of these graphs immerse $D_m$ we may assume $d_{G^*}(y_i) \ge 2$ for $i=0,1$.  If $d_{G^*}(x_0) = 0$ then $G$ must immerse the rightmost graph in Figure \ref{figy1x0}, so again we have $D_m \prec G$.  
	
	\begin{figure}[h]
		\centering
		\includegraphics[height=2.5cm]{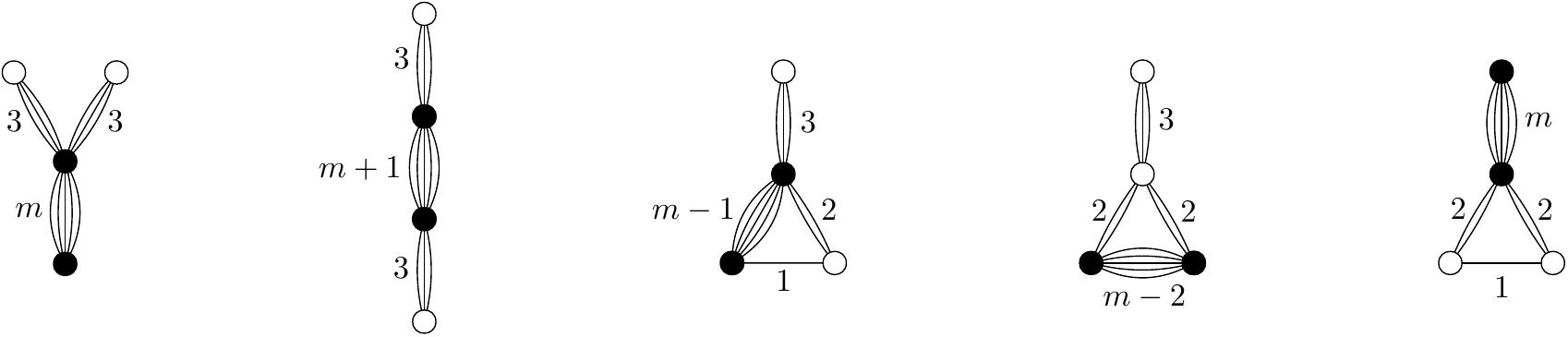}
		\caption{Immersions in $G$ when $d_{G^*}(y_i) = 1$ or $d_{G^*}(x_i) = 0$}
		\label{figy1x0}
	\end{figure}
	
	At this point, we have shown $d_{G^*}(y_i) \ge 2$ and $d_{G^*}(x_i) \ge 1$ for $i=0,1$ and there are just three possibilities for the graph $G^*$ (up to interchanging the names of the roots $x_0$, $x_1$ and the names of the non-roots $y_0$ and $y_1$).  They are the three graphs shown in the Figure \ref{g_star_cases} and will be handled in separate cases.  
	
	\begin{figure}[h]
		\centering
		\includegraphics[height=2.8cm]{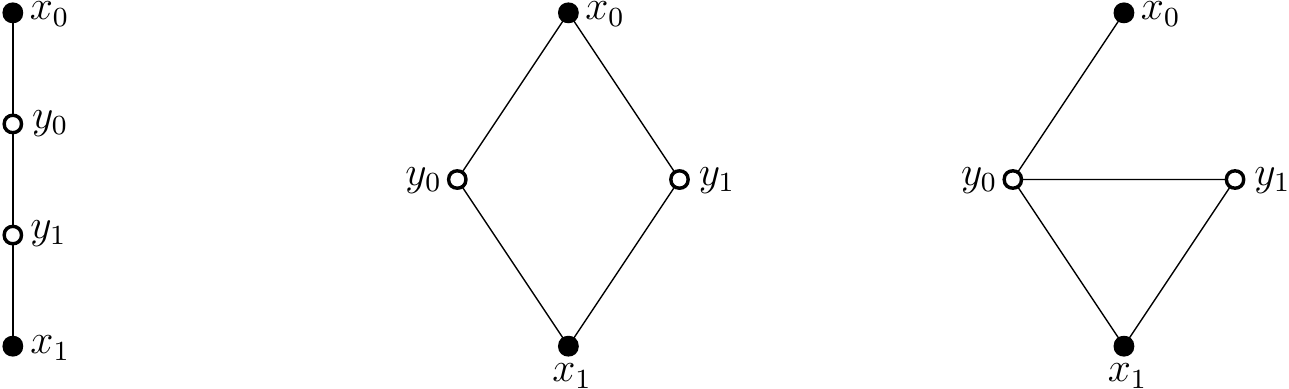}
		\caption{Possibilities for $G^*$}
		\label{g_star_cases}
	\end{figure}
	
	In all three of our cases, we shall classify the paths between $x_0$ and $x_1$ into a small number of types and these are indicated in Figure \ref{pathtypesdm}.  For instance, for the rightmost graph we say that a path is type $\alpha$ if it has vertex sequence $x_0,x_1$, type $\beta$ if it has vertex sequence $x_0, y_0, x_1$, and type $\gamma$ if it has vertex sequence $x_0, y_0, y_1, x_1$.  Now in all three cases, we choose a maximum cardinality packing of edge-disjoint paths from $x_0$ to $x_1$ say $P_1, P_2, \ldots, P_k$ and we let $a$ ($b$, $c$) denote the number of these paths of type $\alpha$ ($\beta$, $\gamma$).  Note that $\lambda_s(G) \ge m$ implies $k \ge m$.  An edge $e \in E(G) \setminus \left (\bigcup_{i=1}^k E(P_i) \right )$ is called an \emph{extra} edge.  Note that there are no extra edges of type $\alpha$ since this would contradict the maximality of our packing.

	\begin{figure}[h]
		\centering
		\includegraphics[height=3cm]{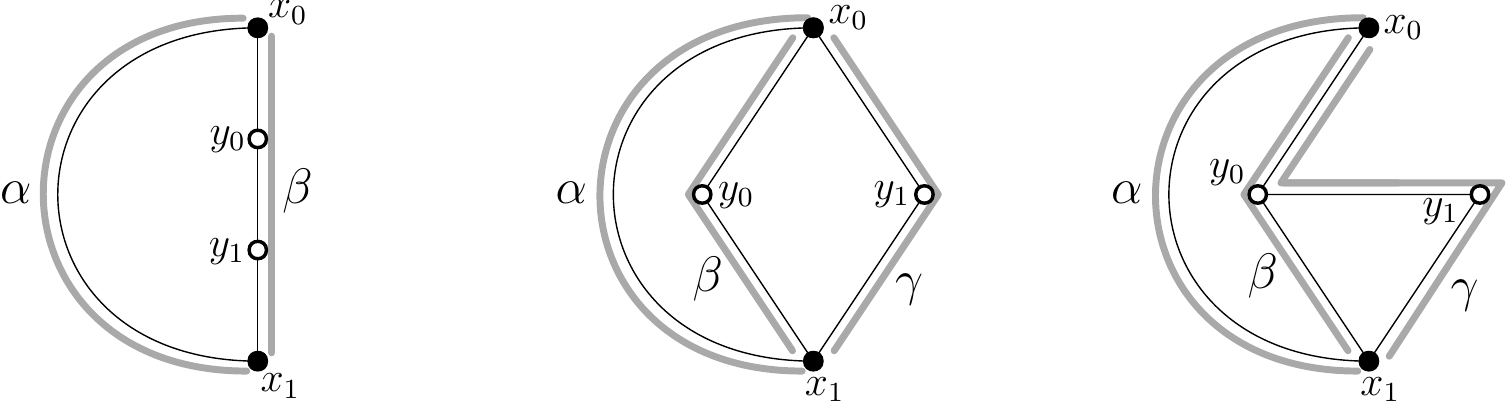}
		\caption{Types of Paths}
		\label{pathtypesdm}
	\end{figure}

	\bigskip
	
	\noindent{\it Case 1: } $G^*$ is the leftmost graph in Figure \ref{g_star_cases}
	
	\smallskip
	
	We note that $b \ge 1$ and split into subcases based on $b$.  First  suppose that $b=1$.  If $e(y_0, y_1) = 1$ then $d(y_0), d(y_1) \ge 3$ forces the existence of extra edges $x_0 y_0$ and $x_1 y_1$.  Now $\lambda_s^i(G) \ge m+1$ implies $a \ge m$ and we find $D_m \prec G$.  If $e(y_0, y_1) \ge 2$ then the maximality of our packing implies $\min\{ e(x_0, y_0), e(x_1, y_1) \} = 1$, but then $\lambda_n^i(G) \ge 4$ implies $\max\{ e(x_0, y_0), e(x_1, y_1) \} \ge 3$.  Since $a \ge m-1$ we again find that $D_m \prec G$.  Next suppose that $b = 2$.  If $a \ge m$ then we have $D_m \prec G$.  If $a = m-1$ then $G$ will be type $B_m$ if there are no extra edges, and $D_m \prec G$ if there is an extra edge.  If $a=m-2$ then $\lambda^i_s(G) \ge m+1$ implies that there is an extra edge $y_0 y_1$ and we have $D_m \prec G$.  Finally, we suppose $b \ge 3$.  If $k = a+b \ge m+1$ then $D_m \prec G$.  Otherwise $a+b = m$ and $\lambda_s^i(G) \ge m+1$ implies that there is an extra edge $y_0 y_1$ and again we have $D_m \prec G$.
	
	\bigskip
	
	\noindent{\it Case 2: } $G^*$ is the middle graph in Figure \ref{g_star_cases}
	
	\smallskip
	
	First note that $b,c \ge 1$.  If $k=m$ then $\lambda_s^i(G) \ge m+1$ implies that there is an extra edge of the form $x_0 y_0$ or of the form $x_1 y_1$, and similarly there is an extra edge of the form $x_0 y_1$ or of the form $x_1 y_0$.  We cannot have both $x_0 y_i$ and $y_i x_1$ as extra edges, since this would contradict the maximality of our packing, but then we have $D_m \prec G$.  Accordingly, we may now assume $k \ge m+1$.  Now we split into subcases depending on the values of $b,c$.  If $b=c  = 1$ then $d(y_i) \ge 3$ implies the existence of an extra edge incident with $y_i$ for $i=1,2$ and this gives $D_m \prec G$.  If $b=1$ and $c\ge 2$, then $d(y_0) \ge 3$ implies the existence of an extra edge of the form $x_0 y_0$ or $x_1 y_0$ and in either case we have $D_m \prec G$.  A similar argument handles the case $b \ge 2$ and $c=1$.  Finally, we consider the case $b,c \ge 2$.  If $k \ge m+2$ then $D_m \prec G$.  If $m=k+1$ then $G$ has type $B_m$ if there are no extra edges, and $D_m \prec G$ if there is an extra edge.
	
	\bigskip
	
	\noindent{\it Case 3: } $G^*$ is the rightmost graph in Figure \ref{g_star_cases}
	
	\smallskip
	
	First suppose that $e(x_0, y_0) = 1$ and note that this implies $a = e(x_0, x_1) \ge m-1$.  If $e(x_1, y_1) \ge 2$ then $D_m \prec G$.  Otherwise it follows from $\lambda_n(G) \ge 3$ that $e(x_1, y_0) , e(y_0, y_2) \ge 2$ and again  we get $G \succ D_m$.  Therefore, $e(x_0, y_0) \ge 2$ and we may assume (without loss) that in our maximum packing we have $b,c \ge 1$.  First suppose that $c=1$.  If there is an extra edge $y_0 y_1$ then (using $b \ge 1$) we find $D_m \prec G$, otherwise $d(y_1) \ge 2$ implies an extra edge of the form $y_1 x_1$.  This extra edge gives an immersion of $D_m$ unless $k = m$ so we may now assume this.  However, if $k=m$ it follows from $d( \{y_1, x_1 \}) \ge m+1$ that there exists an extra edge which forces an immersion of $D_m$.  So we may now assume $c \ge 2$ and then we are immediately finished if $k \ge m+1$ since a path of type $\gamma$ together with a $x_0x_1$ path may be traded for an edge between $x_0$ and $y_1$.  In the only remaining case $k=m$, and it follows from $d( \{y_1, x_1 \}) \ge m+1$ that there is an extra $x_1y_0$ edge which forces immersion of $D_m$.  This completes the proof.
\end{proof}
\subsection{Graphs on at least five vertices}
In this subsection, we prove the main result of this section. Let us start by making a couple observations about graph segmentations.
\begin{observation}
	\label{seg-obs}
	Suppose that a graph $G$ has a segmentation of width $k$ relative to $(X, Y)$.
	\begin{enumerate}
		\item
		\label{dx0m}
		If $|X| = 2$ and $x \in X$ has $d(x) = k$, then $G$ has a segmentation of width $k$ relative to $(\{x\}, Y)$.  
		\item 
		\label{no-ht-even}
		Every $v \in V(G) \setminus (X\cup Y)$ has $d(v)$ even.
	\end{enumerate}	
	
\end{observation}
\begin{proof}
	Let $X \subset X_1 \subset \ldots \subset X_n$ be a segmentation of width $k$ relative to $(X,Y)$ (so $Y = V(G) \setminus X_n$).  For the first part, observe that under the assumption $d(x) = k$, the segmentation $\{x\} \subset X \subset \ldots \subset X_n$ is also width $k$.  For the second part, we may assume that $\{v\} = X_i \setminus X_{i-1}$ for some $2 \le i \le n$.  Now $d(X_i) = k = d(X_{i-1})$ implies $e(v, X_{i-1}) = e(v, V(G) \setminus X_i)$ and thus $d(v) = 2 e(v,X_{i-1})$ is even.  
\end{proof}

We are now going to prove the harder direction of Theorem \ref{Dm-thm}. If $H$ is a graph with roots $(x_0, \ldots , x_k)$ and $X \subset V(H)$ contains a single root vertex, say $x_0$, (whenever convenient) we will interpret $H.X$ as a rooted graph with roots $(X, x_1, \ldots , x_k)$.

\begin{proof}[Proof of the `only if' direction of Theorem \ref{Dm-thm}]
	Towards a contradiction, suppose $G = (V, E)$ is a counterexample to the Theorem for which $|V| + |E|$ is minimum. We will prove some properties of $G$ which finally shows that $G$ does not exist.
	
	\begin{enumerate}[label=(\arabic*), labelindent=0em ,labelwidth=0cm, parsep=6pt, leftmargin =7mm]
		\item 
		\label{Dm-ec}
		$\lambda_s^i (G) \ge m+1$.
		
		Suppose (for a contradiction) that there exists an edge-cut $\delta(U)$ of size at most $m$ in $G$ which separates $x_0 \in U$ from $x_1$, where $|U|, |V\setminus U| \ge 2$. Since $\lambda_s (G) \ge m$, we have $d(U) =m$. Define the graphs $G_0 =G. (V\setminus U), G_1 = G.U$, and let $y_0, y_1$ be the new vertices resulting by identifying $V\setminus U, U$ to a vertex, respectively (so $d_{G_i}(y_i) =m$). For $i=0,1$, let $x_i, y_i$ be root vertices of $G_i$. By the edge-connectivity of $G$, $G \succ_r G_i$, and thus $G_i \nsucc_r D_m$. If $|V(G_i)| \ge 4$, then (by the minimality of our counterexample $G$) the theorem implies that $G_i$ is type $A_m$ or type $B_m$. However, from $d_{G_i}(y_i) = m$ we deduce that $G_i$ is type $A_m$. So, there exists $X_i \subseteq V(G_i)$ with $x_i \in X_i$ and $|X_i| \le 2$ so that the graph $G_i$ has a $(2, 1)$- segmentation of width $m$. On the other hand, if $|V(G_i)| \le 3$, for $X_i = V(G_i) \setminus y_i$ we have $x_i \in X_i$, $|X_i| \le 2$, and the graph $G_i$ has a has a $(2, 1)$- segmentation (of length 1) of width $m$. It now follows that the original graph $G$ has type $A_m$ relative to $X_0$ and $X_1$, and this contradiction establishes \ref{Dm-ec}.
		
		\item 
		\label{Dm-no-greedy-root}
		For every vertex $v \in V \setminus \{x_0, x_1\}$ we have $d(v) > 2 e(v, x_i)$, for $i = 0, 1$.
		
		Suppose (for a contradiction) that $v \in V \setminus \{x_0, x_1\}$ exists with, say, $d(v) \le 2 e(v, x_0)$. Let $G'$ be the graph obtained from $G$ by identifying $\{x_0, v\}$ to a new root vertex, and note that Lemma \ref{atleast4} implies that $|V(G')| \ge 4$. On the other hand, it follows from $d(v) \le 2 e(v, x_0)$ that $G\succ G'$, and thus $G'\nsucc D_m$. Now, $G$ being a minimum counterexample implies that the Theorem holds for $G'$, and so it is either type $A_m$ or type $B_m$. It follows from \ref{Dm-ec} that $G'$ is not type $A_m$. It is now straightforward to check that $G$ satisfies the theorem.
%
%
		\item 
		\label{Dm-no-dU4}
		There does not exist $i=0, 1$ and $U \subset V\setminus \{x_0, x_1\}$ for which $|U| \ge 2$ and $d(U) =4$, and $e(x_i, U) \ge 2$.
		
		Towards a contradiction, suppose such $U$ exists with, say, $e(x_0, U)\ge 2$. Choose distinct $e, e' \in E(x_0, U)$, and let $\{f, f'\} = \delta(U) \setminus \{e, e'\}$. Now let $G'$ be the graph obtained from $G$ by subdividing $f, f'$ with a new vertex, and then identifying the two new vertices of degree two to a new vertex $y$.
		
		Consider $H = G'[U \cup \{x_0, y\}]$, rooted at $x_0, y$. By construction, $d_H(x_0) = d_H(y) =2$, and we have $|E(H)| < |E|$. If $\lambda_s(H) < 2$, it follows from $\lambda^i(G) \ge 4$ that $|U| =2$, and that both vertices in $U$ have degree three, and thus there exist $u \in U$ with $e_G(u, x_0) =2$. This, however, contradicts \ref{Dm-no-greedy-root}, and thus $\lambda_s(H) \ge 2$. Moreover, $\lambda_n(H) \ge \lambda_n(G) \ge 3$, and  $\lambda^i_n(H) \ge \lambda^i_n(G) \ge 4$. So, we can apply the Theorem to $H$ to conclude that either $H \succ_r D_2$ or $H$ has type $A_2$ or type $B_2$. Since the root vertices of $H$ have degree two, $H$ must be type $A_2$, and (by part \ref{dx0m} of Observation \ref{seg-obs}) it is a doubled path. This, however, implies that there exists $u\in U$ with $d(u) = 4$ and $e_G(u, x_0)$, which contradicts \ref{Dm-no-greedy-root}. Thus, $H \succ_r D_2$.
		
		Next, let $K$ be the graph obtained from $(G. U)\setminus \{e, e'\}$ by adding a new vertex $z$ which has two edges to $U$ and $m-2$ edges to $x_0$. It follows from $d(U) = 4$ and  $\lambda(G)\ge m$ that there are $m$ edge-disjoint $z-x_1$ paths in $K$. This, together with $H \succ_r D_2$ implies that $G \succ_r D_m$---a contradiction.
	\end{enumerate}
	Before proceeding, let us introduce some helpful notation. We call an edge $e\in E$ {\it safe} if the graph $G'$ obtained from $G\setminus e$, followed by suppressing any resulting degree two vertices, satisfies $|V(G')| \ge 4$, and $d_{G'}(x_0), d_{G'}(x_1) \ge m$. Observe that if $e$ is safe, \ref{Dm-ec} implies that $\lambda_s(G') \ge m$. Moreover, it follows from $\lambda^i_n(G) \ge 4$ that $\lambda_n(G') \ge 3$. Below, we will confirm that we also have $\lambda^i_n (G') \ge 4$, which then puts us in a position to apply th e Theorem to it.
	\begin{enumerate}[label=(\arabic*), labelindent=0em ,labelwidth=0cm, parsep=6pt, leftmargin =7mm]
		\setcounter{enumi}{3}
		\item
		\label{Dm-safe}
		If $e \in E$ is safe, and $G'$ is the graph obtained from $G\setminus e$ by suppressing degree two vertcies, then $\lambda^i_n (G') \ge 4$.
		
		Suppose (for a contradiction) that $\lambda_n^i (G') =3$. Let $U_1, \ldots , U_k$ be the maximal subsets of $V(G') \setminus \{x_0, x_1\}$ for which $|U_i| \ge 2$ and $d_{G'} (U_i) =3$. Note that for $1 \le i <j \le k$ the sets $U_i$ and $U_j$ are distinct. It is because otherwise $d(U_i\cap U_j) + d(U_i \cup U_j) \le d(U_i) + d(U_j) =6$, which together with $\lambda_n(G') \ge 3$ would imply that $d(U_i \cup U_j) =3$---contradicting maximality of $U_i, U_j$.
		
		Now, let $G''$ be the graph obtained from $G'$ by identification of each set $U_i$ to a new vertex $u_i$. Note that \ref{Dm-no-dU4} implies that $|V(G'')| \ge 4$. Moreover, we have $\lambda_s(G'') \ge m, \lambda_n(G'') \ge 3$ and $\lambda^i_n(G'') \ge 4$. On the other hand, since $\lambda_n(G') \ge 3$ we have $G' \succ G''$, and thus $G'' \nsucc_r D_m$ (else $G\succ_r G'\succ_r D_m$). Since $|V(G'')|< |V|$, we can apply the Theorem to $G''$ to conclude that it is either type $A_m$ or type $B_m$. Now since every non-root vertex in a graph of type $B_m$ has even degree, it must be that $G''$ has type $A_m$. Now, part \ref{no-ht-even} of Observation \ref{seg-obs} implies that (by possibly relabeling $x_0, x_1$) the graph $G''$ has type $A_m$ relative to $(X_0, X_1)$, where $X_0 = \{x_0, u_1\}$. Furthermore, it follows from $d_{G''}(x_0) \ge m$, $d_{G''} (u_1) = 3$ together with $d_{G''}(\{x_0, u_1\}) =m$ that $e_{G''}(x_0, u_1) \ge 2$. This, however, implies that $e_G(x_0, U_1) \ge 2$, and since $d_G(U_1) = 4$, we get a contradiction with \ref{Dm-no-dU4}. This completes the proof of \ref{Dm-safe}.
		\item 
		\label{Dm-all-odd}
		Every vertex $v \in V\setminus \{x_0, x_1\}$ has odd degree.
		
		Towards a contradiction, suppose $v$ violates \ref{Dm-all-odd}. It follows from \ref{Dm-no-greedy-root} that there is a neighbour $u$ of $v$ which is not a root vertex. Since $d(v) \ge 4$, $uv$ is safe. So the graph obtained from $G\setminus uv$ by suppressing degree two vertices satisfies the hypothesis of the theorem. Since $G\succ_r G'$, we have $G'\nsucc_r D_m$, so $G'$ either has type $A_m$ or type $B_m$. Since $d_{G'} (v)$ is odd (by part \ref{no-ht-even} of Observation \ref{seg-obs}) $G'$ must be type $A_m$ relative to $(X_0, X_1)$, with $v$ in $X_0 \cup X_1$, say $X_0  = \{x_0, v\}$. Now $d_{G'} (x_0) \ge m, d_{G'}(\{x_0, v\}) = m$ together with the parity of $d_{G'}(v)$ imply that $2e_{G'}(x_0, v) \ge d_{G'}(v)$. However this implies $2e_{G}(x_0, v) \ge d_{G}(v)$, which is a contradiction with \ref{Dm-no-greedy-root}.
		\item
		\label{Dm-root-tight}
		We have $d(x_0) = d(x_1) = m$.
		
		Suppose (for a contradiction) that $d(x_0)> m$. If $x_0$ has a neighbour $v$ other than $x_1$, let $e = x_0v$; else, let $e = x_0x_1$. In either case $e$ is safe, and consider $G'$ which is the graph obtained from $G\setminus e$ by suppressing degree two vertcies. Since $G \succ_r G'$, we have $G'\nsucc_r D_m$ and since $|E(G')| < |E|$, we can apply the Theorem to $G'$. Since $G'$ has non-root vertices of odd degree, it has type $A_m$ relative to $(X_0, X_1)$. As before, it follows from part \ref{no-ht-even} of Observation \ref{seg-obs}  that either $|V(G')| = 4$, and both non-root vertices of $G'$ have odd degree, or $|V(G')| = 5$ and there is a unique non-root vertex of $G'$ which has even degree. In either case, we may assume $X_0 = \{x_0, u\}$. It now follows from $d_{G'}(x_0) \ge m$ and $d_{G'}(X_0) = m$ that $2e_{G'}(x_0, u) \ge d_{G'}(u)$. Now, note that $e$ must be in $\delta(X_0)$, and thus we get $2e_{G}(x_0, u) \ge d_{G}(u)$, which is a contradiction with \ref{Dm-no-greedy-root}.
	\end{enumerate}
	We can now finish the proof. Note that it follows from \ref{Dm-all-odd} and \ref{Dm-root-tight} that $|V| \ge 6$. Let $v\in V\setminus \{x_0, x_1\}$ and note that by \ref{Dm-no-greedy-root}, we may choose an edge $e = vv'$, where $v' \notin \{x_0, x_1\}$. The edge $e$ is safe, and as in the proof of \ref{Dm-all-odd} and \ref{Dm-root-tight} the graph $G'$ obtained from $G\setminus e$ by suppressing degree two vertcies has type $A_m$. Since $|V| \ge 6$, there exist at least two vertices $w, w' \in V\setminus \{x_0, x_1, v, v'\}$ and it follows from \ref{Dm-all-odd} that $w, w'$ have odd degree (in both $G$ and $G'$). Therefore we may assume $G'$ has type $A_m$ relative to $(X_0, X_1)$, where $X_0 = \{x_0, w\}$. However, then we have $m = e_{G'}(X_0) = e_{G}(X_0)$ which contradicts \ref{Dm-ec}. This final contradiction completes the proof of Theorem \ref{Dm-thm}.
\end{proof}

\subsection{Forbidding a (rooted) $K_4$ immersion}
\label{sec-k4}

Since $K_4$ is a graph of special interest, we have devoted this section to characterizing the graphs which do not immerse $K_4$ with two, one, or zero roots.  Since all of these results are convenient to state without the assumption of internal 4-edge-connectivity, we have done so.  

\begin{corollary}
	\label{D3-wo*}
	Let $G$ be a $3$-edge-connected graph where $|V(G)| \ge 4$, with two root vertices. Then $G\nsucc_r D_3$ iff either of the following occurs:
	\begin{itemize}
		\item
		$G$ is a doubled cycle.
		\item
		$G$ has a segmentation of width 3 relative to $(X_0, X_1)$ with the added property that if $|X_i| \ge 3$ then $X_i$ does not contain a root vertex.
	\end{itemize}
\end{corollary}

\begin{proof}
	Suppose $G\nsucc_r D_3$. If $G$ does not have an internal $3$-edge-cut with both roots on the same side, we apply Theorem \ref{Dm-thm} to conclude that $G$ has type $A_3$ or type $B_3$. In the former case there is nothing left to prove, and in the latter, note that it follows from our definition that a graph of type $B_3$ is a doubled cycle.
	
	 In the remaining case, let $U_1, \ldots, U_k$ be the maximal subsets of $V(G) \setminus \{x_0, x_1\}$ for which $|U_i| \ge 2$ and $d_{G} (U_i) =3$. A similar argument as in the proof of \ref{Dm-safe} in the proof of Theorem \ref{Dm-thm} shows that for $1\le i<j \le k$ we have $U_i \cap U_j = \emptyset$. Now, let $H$ be the graph obtained from identifying each $U_i$ to a new vertex $u_i$ (so $\lambda^i_n(H)\ge 4$). Since $G$ is $3$-edge-connected, $G \succ_r H$, and thus $H \nsucc_r D_3$. Since $d_{H}(u_1) =3$, part \ref{no-ht-even} of Observation \ref{seg-obs}  implies that $H$ has type $A_3$ relative to $(X_0, X_1)$, where (by possibly relabeling $x_0, x_1$) $X_0 = \{x_0, u_1\}$ and $X_1 = \{x_1, v\}$, where $v = u_2$ if $u_2$ exists. This implies that if $U_2$ exists, $G$ has a segmentation of width three relative to $(U_1, U_2)$, and otherwise $G$ has a segmentation of width three relative to $(U_1, X_1)$, as desired.
\end{proof}

%

Corollary \ref{D3-wo*} gives the structure of $3$-edge-connected graphs excluding a $D_3$ immersion, or equivalently an immersion of $K_4$ with two roots. Our next task is to characterize the structure of $3$-edge-connected graphs excluding rooted immersion of $K_4$ with one or no root. 

\begin{corollary}
	\label{cor-rk4}
	Let $G =(V, E)$ be a $3$-edge-connected graph with up to one root vertex and $|V| \ge 4$. Then $G$ does not have an (a rooted) immersion of $K_4$ if and only if:
	\begin{itemize}
		\item
		$G$ is a doubled cycle, or
		\item
		$G$ has a $(2, 2)$-segmentation of width three.
	\end{itemize}
\end{corollary}

\begin{proof}
We proceed by induction on $|V|$.  If $G$ has a root vertex, denote it by $x$; otherwise, let $x$ be an arbitrary vertex.  Choose a vertex $y \in V \setminus \{x\}$, consider both $x$ and $y$ as roots and apply the previous corollary.  If we find a rooted immersion of $D_3$ or discover that $G$ is a doubled cycle, we are finished.  Otherwise, $G$ has a segmentation of width 3 relative to $(X_0, X_1)$ where each $X_i$ either has size $\le 2$ or does not contain a root vertex.  If $|X_0|, |X_1| \le 2$ we have nothing left to prove.  So suppose $|X_0| \ge 3$ and form the graph $G_0 = G. (V \setminus X_0)$ where the vertex formed by this identification, denoted $x_0$, is considered as the root.  If $G_0$ with root vertex $x_0$ has a rooted immersion of $K_4$, then it follows from the 3-edge-connectivity of $G$ that $G$ (with root vertex $x$) has a rooted immersion of $K_4$ otherwise by induction $G_0$ has a $(2,2)$-segmentation of width 3 (it cannot be a doubled cycle since $d_{G_0}(x_0) = 3$).  However, since $d_{G_0}(x_0) = 3$, Observation \ref{seg-obs} implies that $G_0$ has a $(1,2)$ segmentation of width 3 relative to $( \{x_0\}, Y_0)$.  By applying a similar argument if $|X_1| \ge 3$, we conclude that the original graph $G$ has a $(2,2)$-segmentation, as desired.
\end{proof}

\section{Forbidding a rooted $W_4$ immersion}
\label{sec-rw4}
In this section, we state and prove our structural theorem for graphs without an immersion of rooted $W_4$, which is $W_4$ where its center is declared to be the root vertex. The result, not only interesting by itself, but also implies a characterization of graphs without $W_4$ immersion. A perhaps more significant application of this theorem appears in our subsequent paper when obtaining a precise structural theorem for graphs which exclude $K_{3,3}$ immersion.

\subsection{Statement of the main theorem}
 To state our result on characterization of graphs without a rooted $W_4$ immersion, we need to introduce four families of such graphs. Let $G$ be a graph with the root vertex $u$. Then we say 
\begin{description}
	\item[Type 1.] $G$ has type $1$
	 if it has a $(2, 3)$-segmentation of width four in which $u$ is in the head of the segmentation.
	
	If $U, W$ are the head and tail of such a segmentation, respectively, we may say $G$ has type $1$ relative to $(U, W)$.
	
	\item[Type 2.] $G$ is type $2$ 
	if there exists a set $W \subset V(G) \setminus \{u\}$ with $1\le |W| \le 2$ so that the graph $G^*$ obtained by identifying $W$ to a single vertex $w$ has a doubled cycle $C$ containing $u, w$ satisfying one of the following:
	\begin{description}
		\item[(2A)] $u$ and $w$ are not adjacent in $C$ and $G^* = C + uw$
		\item[(2B)] $u$ and $w$ have a common neighbour $v$ in $C$ and $G^* = C + uv + vw$
		\item[(2C)] $u$ and $w$ are adjacent in $C$ and $G^* = C + uw$
	\end{description}
	\begin{figure}[htbp]
		\centering
		\begin{subfigure}[b]{0.3\textwidth}
			\centering
			\includegraphics{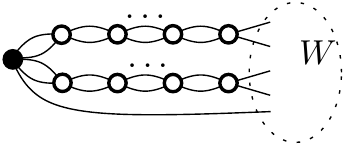}
			\caption{Type 2A}
			\label{fig:2A}
		\end{subfigure}
		\hfill
		\begin{subfigure}[b]{0.3\textwidth}
			\centering
			\includegraphics{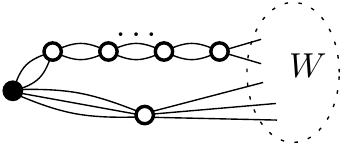}
			\caption{Type 2B}
			\label{fig:2B}
		\end{subfigure}
		\hfill
		\begin{subfigure}[b]{0.3\textwidth}
			\centering
			\includegraphics{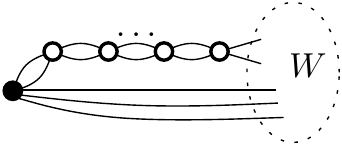}
			\caption{Type 2C}
			\label{fig:2C}
		\end{subfigure}
		\caption{}
		\label{}
	\end{figure}
	\item[Type 3.] $G$ is type $3$ if after sausage reduction $G$ is isomorphic to a graph in Figure \ref{fig:rw4-t3}. That is $G$ is type $3$ if it can be obtained from a graph in Figure \ref{fig:rw4-t3} by replacing the pair of green vertices with a chain of sausages of arbitrary order $\ge 2$.
	\begin{figure}[htbp]
		\centering
		\includegraphics{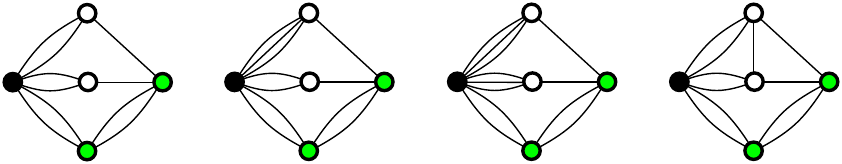}
		\caption{Type 3 graphs after sausage reduction}
		\label{fig:rw4-t3}
	\end{figure}
	\item[Type 4.] $G$ is type $4$ if either
	\begin{itemize}
		\item
		it can be obtained from the leftmost graph in Figure  $\ref{fig:rw4-t4}$ by adding up to one more edge incident to each of $y,y',y''$ in parallel to an existing edge. 
		
		\item
		it is isomorphic to one of the four rightmost graphs in Figure $\ref{fig:rw4-t4}$.
	\end{itemize}
	
	\begin{figure}[htbp]
		\centering
		\includegraphics{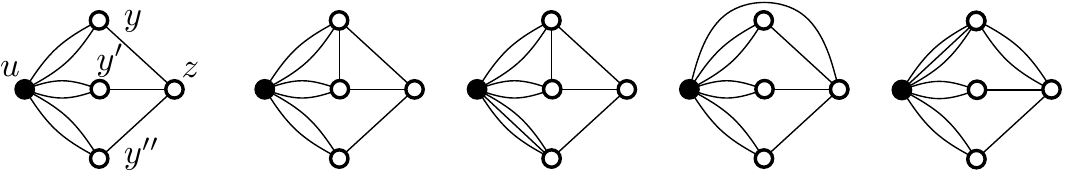}
		\caption{Type 4 graphs}
		\label{fig:rw4-t4}
	\end{figure}
\end{description}

With these definitions, our main result in this section can be stated as follows:
\begin{theorem}
	\label{rw4-thm}
	Let $G$ be a $3$-edge-connected, internally $4$-edge-connected graph with $|V(G)| \ge 5$ and with a root vertex $x$.  Then $G$ contains a rooted immersion of $W_4$ if and only if $G$ does not have one of the types $1$, $2$, $3$, or $4$.
\end{theorem}
\subsection{Proof of the `if' direction}
Before proving the easier part of the theorem, we make a couple simple observations.

\begin{observation}
	\label{containw4obs}
	Suppose that $G$ is a graph, rooted at $x$, which has a rooted immersion of $W_4$ on terminals $T$.
	\begin{enumerate}
		\item \label{t1now4} If $G$ has a segmentation $X_0 \subset X_1 \subset \ldots \subset X_k$ of width four with $x \in X_0$ and $|X_0| \le 2$, then $T \cap X_k = \{x\}$.
		\item \label{w4imparity} If $v \in T\setminus x$ has $d(v)$ even, then there is an edge $e$ incident with $v$ so that $G-e\succ_r W_4$.
		\item \label{sausageok} The graph obtained from $G$ by sausage reducing it has a rooted immersion of $W_4$.
	\end{enumerate}
\end{observation}

\begin{proof}
	The first part follows from the fact that for every set $X \subset V(W_4)$ with $|X| = 2$ which contains the center we have $d(X)= 5$. The second part is immediate from our definitions, and the last follows from part \ref{w4imparity} and the edge-connectivity of $W_4$.
\end{proof}

\begin{proof}[Proof of the `if' direction of Theorem \ref{rw4-thm}]
	We show that graphs of types 1, 2, 3, or $4$ do not immerse rooted $W_4$. For graphs of type 1, this is immediate from part \ref{t1now4} of the previous observation. To verify this for graphs of type 3, by part \ref{sausageok} of the previous observation we only need to show that graphs in Figure \ref{fig:rw4-t3} do not have a rooted immersion of $W_4$. This, as well as verifying the statement for graphs of type 4, is easy enough to do by hand, but we have used a computer to do so.\footnote{The code is available at @}
	
	So let $G$, rooted at $x$, be type 2 relative to $W$, and suppose (for a contradiction) that $G\succ_r W_4$ with $T$ as terminals. By part \ref{sausageok} of the above observation, we may assume $G$ is sausage reduced. Suppose that there is a chain of sausages $G[\{y, z\}]$ in $V(G)\setminus (W\cup x)$. Note that $|\{y, z\}\cap T| \le 1$, otherwise it follows from part \ref{w4imparity} of the previous observation and the internal $4$-edge-connectivity of $W_4$ that the graph $G'$ obtained from $G$ by deleting one copy of the edge $yz$ contains a rooted immersion of $W_4$. This, however, is impossible since $G'$ has type 1.	Therefore, if we let $G''$ be the graph obtained from $G$ by splitting off any chain of sausages  disjoint from $W \cup x$ (if present) down to only one vertex, then $G''\succ_r W_4$. This, immediately gives a contradiction in the cases where $G$ has type 2C, or $|W|=1$. In other cases, $|W|=2$, and $|V(G'')| = 5$, so every vertex in $G''$ is a terminal of $W_4$. Then,  by part \ref{w4imparity} of the above observation, an edge incident to each vertex in $V(G'')\setminus (W \cup x)$ may be removed while an immersion of $W_4$ is preserved. However, in the resulting graph either $d(x)<4$, or there is an internal 3-edge-cut, so this is impossible.
\end{proof}

\subsection{Four edge cuts}
We now embark on proving the `only if' direction of the main result of this section, which will be done through a series of lemmas. This subsection concerns internal 4-edge-cuts in a minimum counterexample to the main theorem.  Our goal will be to show that in a minimum counterexample $G=(V,E)$, every edge-cut $\delta(X)$ with $|X|, |V \setminus X| \ge 3$ satisfies $d(X) \ge 5$.  We begin by considering the case $|X|=2$.  

%
\begin{lemma}
	\label{w4-near-root}
	Let $G=(V,E)$ with root vertex $x$ be a counterexample to Theorem \ref{rw4-thm} with  $|V|$ minimum.  Then $d( \{x, y \} ) \ge 5$ for every $y \in V \setminus \{x\}$.  
\end{lemma}
\begin{proof}
	Suppose (for a contradiction) this is false and choose $y \in V \setminus \{x\}$ so that $d(\{x, y\}) < 5$. Let $X = \{x, y\}$ and note that the internal $4$-edge-connectivity of $G$ implies $d(X) = 4$.  If $|V(G)| = 5$, then $G$ has a $(2, 3)$-segmentation of width four relative to $(X, V\setminus X)$, and thus $G$ has type $1$. So we must have $|V(G)| \ge 6$. Consider $G' = G. X$ (with root $X$), and note that it follows from internal $4$-edge-connectivity that $G$ that $G\succ_r G'$, so $G'\nsucc_r W_4$. By minimality of $G$, the Theorem holds for $G'$, so $G'$ must have type 1, 2, 3, or 4. Since $X\in V(G')$ has degree four, $G'$ can only be type $1$, and moreover by Observation \ref{seg-obs} we may assume $G'$ has a $(1, 3)$-segmentation of width four relative to $( \{X\}, W)$, for some $W\subset V(G')$. It now follows that $G$ has type 1 relative to $(X, W)$---a contradiction.
\end{proof}

\begin{lemma}
	\label{w4-4ecut-Xc-sausage}
	Let $G=(V,E)$ with root vertex $x$ be a counterexample to Theorem \ref{rw4-thm} with  $|V|$ minimum.  
	Let $\delta(X)$ be a $4$-edge-cut in $G$ with $x \in X$ and $|X|\ge 3$. If $|V\setminus X| \ge 3$, then $G[V\setminus X]$ is a chain of sausages.
\end{lemma}
\begin{proof}
	Let $G'= G. (V\setminus X)$, and denote by $y$ the vertex resulting from identifying $V\setminus X$. Declare $G'$ to be rooted at $(x, y)$. First, we show that $G' \succ_r D_4$. Note that $G'$ is internally $4$-edge-connected, and $|V(G')| \ge 4$. So, if $G'\nsucc_r D_4$, it follows from Theorem \ref{Dm-thm} that $G'$ has type $A_4$ or type $B_4$. However, since $d_{G'} (y)=4$, $G'$ must be type $A_4$. So there exists $U\subset V(G')$ such that $x\in U$, $|U| = 2$ and $G'$ has type $A_4$ relative to $U, \{y\}$. This in particular implies that $d_{G}(U) =d_{G'}(U)=4$ which contradicts the previous lemma, and thus $D_4 \prec_r G'$.
	
	Fix a rooted immersion of $D_4$ with roots $(x_0,x_1)$ in $G'$ with roots $(x,y)$.  Let $P,P'$ be the paths of $G'$ corresponding to the two edges of $D_4$ between $x_1$ and $x_0$ and let $Q,Q'$ be the paths of $G'$ corresponding to the two edges of $D_4$ between $x_1$ and the non-root vertices.  Define the edges $\{e, e'\}= \delta_{G'} (y) \cap (E(P) \cup E(P'))$, and  $\{f, f'\}=  \delta_{G'} (y)\cap (E(Q) \cup E(Q'))$. Now let $G''$ be the graph obtained from $G$ by subdividing $e, e'$ ($f, f'$) with a new vertex, and then identifying the degree two vertices to a new vertex, $a$ ($b$), see Fig. \ref{fig:d4d2rw4-Gz}.
\begin{figure}[htbp]
	\centering
	\begin{subfigure}[b]{0.35\textwidth}
		\centering
		\includegraphics[height=2.5cm]{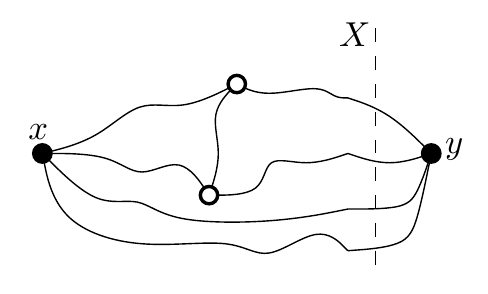}
		\caption{$G'\succ_r D_4$}
		\label{fig:d4d2rw4-Gp}
	\end{subfigure}
	\hfill
	\begin{subfigure}[b]{0.3\textwidth}
		\centering
		\includegraphics[height=2.5cm]{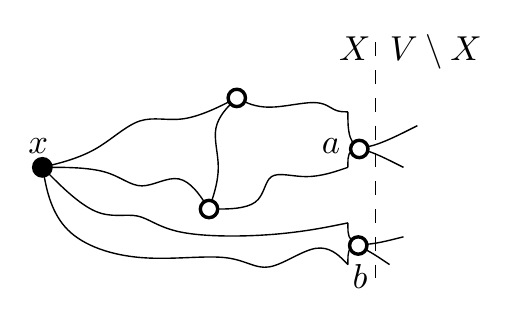}
		\caption{Graph $G''$}
		\label{fig:d4d2rw4-Gz}
	\end{subfigure}
	\hfill
	\begin{subfigure}[b]{0.3\textwidth}
		\centering
		\includegraphics[height=2.5cm]{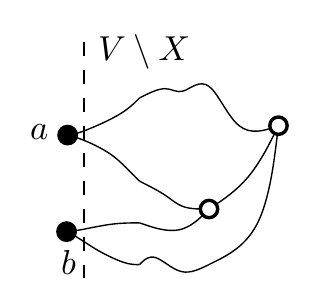}
		\caption{Immersion of $D_2$ in $G^*$}
		\label{fig:d4d2rw4-Gs}
	\end{subfigure}
	\hfill
	\begin{subfigure}[b]{0.66\textwidth}
		\centering
		\includegraphics[height=2.5cm]{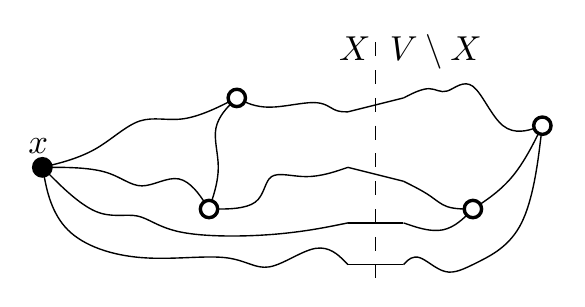}
		\caption{$G' \succ_r W_4$ and $G^* \succ_r D_2$ implies $G\succ_r W_4$}
		\label{fig:d4d2rw4-G}
	\end{subfigure}
	\caption{}
\end{figure}	
%
	
	We define $G^*= G''[(V\setminus X) \cup \{a, b\}]$, with $a, b$ as its root vertices. Observe that (since $G'\succ_r D_4$) if there is a rooted immersion of $D_2$ in $G^*$, then $G\succ_r W_4$ (see Figure \ref{fig:d4d2rw4-G}), so $D_2 \nprec_r G^*$. Note it follows from $|V\setminus X| \ge 3$ and the internal edge-connectivity of $G$ that $G^*$ satisfies the hypothesis of Theorem \ref{Dm-thm} for $m=2$. So, $G^*$ is type $A_2$ or type $B_2$, and since $d_{G^*} (a)=2$, $G^*$ must be type $A_2$. Now since $d_{G^*}(a)= d_{G^*} (b) =2$, $G^*$ is a doubled path.  It follows that $G[V \setminus X]$ is a chain of sausages as desired.
\end{proof}

We now use the Lemma above to prove the main conclusion of this subsection.
\begin{lemma}
	\label{w4-no-deep-4ecut}
	Let $G=(V,E)$ with root vertex $x$ be a counterexample to Theorem \ref{rw4-thm} with  $|V|$ minimum.  
	Every $X \subset V$ with $|X|\ge 2, |V \setminus X| \ge 3$ with $x \in X$ satisfies $d(X) \ge 5$.
\end{lemma}

\begin{proof}
Suppose (for a contradiction) that there exists $X \subset V$ with $x \in X$ satisfying $|X| \ge 2$ and $|V \setminus X| \ge 3$ and $d(x) \le 4$.  It then follows from the previous lemma that the graph $G'$ obtained from $G$ by sausage reduction satisfies $|V(G')| < |V|$.  Note that Lemma \ref{w4-near-root} implies that the root vertex $x$ is not contained in any sausage of $G$ and this resulting graph $G'$ must also satisfy $|V(G')| \ge 5$.  Now $G\succ_r G'$, and thus $G' \nsucc_r W_4$.  By the minimality of the counterexample $G$, we find that $G'$ has one of the types 1, 2, 3, or 4. If $G'$ is type $1$ relative to some $(U, W)$, then $\delta_G(U)$ would be a 4-edge-cut in $G$ which contradicts Lemma \ref{w4-near-root}. Suppose $G'$ is type $2$ relative to $W$ and subject to this $W$ is minimal.  If $G'[Y]$ is a chain of sausages in $G'$ (so, is of order at most two) the minimality of $Y$ implies that $W \cap Y= \emptyset$. Therefore, the reverse of a sausage-shortening applied to $Y$ results in another graph of type 2.  By repeating this argument, we deduce that the original graph $G$ is also type 2---a contradiction.  Note that $G'$ is not isomorphic to one of the graphs in Figure \ref{fig:rw4-t3} either, else $G$ would be type $3$. 	So $G'$ is type $4$, and note that $G'$ has at least one pair of neighbours each of degree four, with two edges between them (since sausage reduction was applied nontrivially to $G$). Thus, $G'$ must be obtained from the leftmost graph in Figure \ref{fig:rw4-t4} by adding one more copy of one edge incident to $z$, say $zy$, and perhaps adding one more copy of $uy'$ and/or $uy''$. In any case it follows that $G$ has type 3, and this final contradiction completes the proof of the Lemma.	
\end{proof}

\subsection{Computation for small graphs}
The main result of this subsection is Lemma \ref{w4-code} which asserts that the Theorem holds for graphs on at most eight vertices. This has been verified using a computer, and our code can be found on the arXiv. The simple observation below shows that the verification needs to be done only for finitely many graphs.
\begin{observation}
	\label{e-mult-cap}
	Let $G, H$ be rooted graphs, and let $u, v \in V(G)$ satisfy $e(u, v) > |E(H)|$. Then $G \succ_r H$ if and only if the graph obtained from $G$ by deleting one copy of $uv$ edge has a rooted immersion of $H$.
\end{observation}
\begin{lemma}
	\label{w4-code}
	If $ G = (V, E) $ is a counterexample to Theorem \ref{rw4-thm} with $|V|$ minimum, then $|V|\ge 9$.
\end{lemma}
\begin{proof}
	Suppose $G$ is a graph with a root vertex $x$, where $d(x) \ge 4$ and $|V (G)| \le 8$. Suppose further that $G$ is $3$-edge-connected, internally $4$-edge-connected, and satisfies the statement of Lemma \ref{w4-no-deep-4ecut}. We will show that $G$ satisfies Theorem \ref{rw4-thm}, i.e. if $G \nsucc_r W_4$, it is either type 2, or is isomorphic to one of the graphs in Figures \ref{fig:rw4-t3}, or \ref{fig:rw4-t4}.
	
	Note that thanks to Observation \ref{e-mult-cap}, the Lemma can get verified through a finite calculation. In fact it follows that to verify the Lemma it suffices to check the finite number of rooted $3$-edge-connected, internally $4$-edge-connected graphs, with edge-multiplicity at most $|E(W_4)|=8$, for which the root vertex $x$ has degree at least four, and for any set $Y \subset V(G)\setminus \{x\}$ with $|Y|\ge 3 $ we have $d (Y) \ge 5$. This calculation is done in Sagemath, and here is a high-level description of the algorithm, which is run for $5 \le n \le 8$.
	\begin{description}
		\item[Step 1.] 	We take the list of connected simple graphs on $n$ vertices, and use it to generate all rooted graphs on $n$ vertices. Then, the rooted graphs which have a rooted immersion of $W_4$ are filtered out.
		\item[Step 2.] 	For any rooted graph $G $ surviving from Step 1, \verb|repair(G)| generates a list consisting of all edge-minimal rooted multigraphs $G'$ such that:
		\begin{itemize}
			\item
			the underlying simple graph of $G'$ is $G$,
			\item
			$d_{G'} (x) \ge 4$, where $x$ is the root vertex of $G'$,			
			\item
			$G'$ is 3-edge-connected and internally 4-edge-connected,
			\item
			for any internal edge-cut $\delta(Y)$ with $x\notin Y$ and $|Y| \ge 3$ we have $d(Y)\ge 5$,
			\item
			$G'$ does not have a rooted immersion of $W_4$.
		\end{itemize}
	\item[Step 3.] Suppose the simple rooted (connected) graph $G$ is such that $\mathcal{G}_1=$ \verb|repair(G)| is nonempty. Then, using $\mathcal{G}_1$, we generate $\mathcal{G}_2 =$ \verb|obstruction(G)| which consists of all rooted multigraphs whose underlying simple rooted graph is $G$, meet the edge-connectivity conditions that the graphs in $\mathcal{G}_1$ satisfy, have edge-multiplicity at most eight, and do not immerse rooted $W_4$.
	\item[Step 4.] Every graph in $\mathcal{G}_2$ is tested if it has type 2, or is isomorphic to one of the graphs in Figures \ref{fig:rw4-t3} or \ref{fig:rw4-t4}.
	\end{description}	
	The calculation is done rather fast. The calculation for every $n \in \{5, 6, 7\}$ took a  desktop computer less than a minute. However, the calculation for $n =8$ took much longer---almost 20 minutes. It took the computer one minute to carry out step 1, i.e. to check the nearly 72,500 connected simple rooted graphs on eight vertices for a $W_4$ immersion, thereby giving a list \verb|N8| of almost 40,000 simple rooted connected graphs on eight vertices which do not immerse rooted $W_4$. Then 20 minutes was spent on carrying out steps 2, 3 for every graph in \verb|N8|. Since no obstruction is found for $n =8$, step 4 is not performed for this case.
\end{proof}
\subsection{Five edge cuts}
The next lemma establishes two local properties for a minimum counterexample. The lemma proves to be quite helpful in handling 5-edge-cuts, as well as in the final part of the proof of the main theorem.
\begin{lemma}
	\label{rw4-local}
	Let $ G = (V, E) $ be a a graph rooted at $x$ which is a counterexample to Theorem \ref{rw4-thm} with $|V|$ minimum. Then 
	\begin{enumerate}[label=(\arabic*)]
		\item 
		\label{w4-no-greedy-nbr}
		There do not exist $v, w\in V\setminus x$ such that $e (v, w)\ge \frac{1}{2} d(w)$.
		\item
		\label{w4-4ecut}
		Suppose $\delta (Y)$ is an internal $4$-edge-cut in $G$, with $|Y| \le |V\setminus Y|$. Then $|Y| = 2$, $x \notin Y$, and both vertices in $Y$ have degree three.	
	\end{enumerate}
\end{lemma}
\begin{proof}
	For part \ref{w4-no-greedy-nbr}, suppose (for a contradiction) that such $v, w$ exist. Let $G'= G. \{v, w\}$, rooted at $x$. Note that $|V(G')| = |V(G)| -1$, and since $e (v, w)\ge \frac{1}{2} d(w)$, we have  $G \succ_r G'$. Therefore  $G'\nsucc_r W_4$, and the Theorem  holds for $G'$. Since $|V(G')|\ge 8$, $G'$ has one of the types 1, 2, or 3. However, $G'$ being type $1$ implies that there is an internal $4$-edge-cut in $G$ with at least three vertices on opposite side of $x$, contradicting Lemma \ref {w4-no-deep-4ecut}. So $G'$ is not type $1$. In a similar manner we conclude that $G'$ is sausage reduced, and since $|V(G')|\ge 8$, $G'$ is not type $2$ or 3 either, a contradiction. (Observe that after sausage reduction a graph of type $2$ or 3 has at most seven vertices). 
	For part \ref{w4-4ecut}, note that it follows from Lemma \ref{w4-no-deep-4ecut} that $|Y| =2$, and $x\notin Y$. Let $Y =\{ u, v\}$. Since $d (Y) =4$, we have $e (u, v) >0$, and part \ref{w4-no-greedy-nbr} implies that $e (u, v) =1$, as desired.
\end{proof}

\begin{lemma}
	\label{w4-no-deep-5ecut}
	Let $ G = (V, E) $ be a graph rooted at $x$ which is a counterexample to Theorem \ref{rw4-thm} with $|V|$ minimum. Then there does not exist a 5-edge-cut $\delta (X)$ with $x\in X$ satisfying $|X|\ge 3$ and $|V\setminus X| \ge 4$.
\end{lemma}
\begin{proof}
	Suppose (for a contradiction) that such a set $X$ exists and let $Y = V\setminus X$. Consider the graph $G' = G.Y$ where $y$ is the vertex formed by identifying $Y$ and treat this as a rooted graph with roots $(x,y)$.  We claim that $G'$ with $(x,y)$ has a rooted immersion of $D_4$.  Otherwise Theorem \ref{Dm-thm} implies that $G'$ has type $A_4$ or type $B_4$. In the former case, we get a contradiction with Lemma \ref{w4-near-root}, and in the latter case, there will be a neighbour $u$ of $x$ with $e(x, u) =2$ and $d(u)=4$, which contradicts Lemma \ref{rw4-local}\ref{w4-no-greedy-nbr}.
	
Since $d(X) = 5$ we may choose an edge $e \in \delta(X)$ so that $D_4 \prec_r (G' \setminus e)$.  Let $H$ be the graph obtained from $G \setminus e$ by suppressing any vertices of degree 2 and let $Y'$ denote the subset of $V(H)$ corresponding to $Y$ (so either $Y' = Y$ or $Y' = Y \setminus \{w\}$ where $d_{G}(w) = 3$ and $w$ is incident to $e$).  Now by an argument similar to that of Lemma \ref{w4-4ecut-Xc-sausage} we find that $H[Y']$ is a chain of sausages (of order at least three).  However, it is easy to check that whether the endpoint of $e$ in $G$ is a vertex present in $H$ or it got suppressed, there is a vertex of degree four in $G$ incident with a pair of parallel edges, and this contradicts  Lemma \ref{rw4-local}\ref{w4-no-greedy-nbr}.
\end{proof}

\subsection{Finishing the proof}
In this subsection, we use Lemmas \ref{w4-code}, \ref{rw4-local}, \ref{w4-no-deep-5ecut} to prove our main result on graphs which exclude an immersion of rooted $W_4$.
\begin{proof}[Proof of Theorem \ref{rw4-thm}]
	Suppose (for a contradiction) that the Theorem is false, and let $G = (V, E)$ be a counterexample to the Theorem so that
	\begin{enumerate}[label=(\roman*)]
		\item $|V|$ is minimum.
		\item $|E|$ is minimum subject to (i).
	\end{enumerate}
	First, we show that $G$ is simple. Otherwise, there exist adjacent vertices $u, v$ such that $e (u, v) \ge 2$. Let $G'$ be the graph obtained from  $G$ by deleting one copy of $uv$. Note that Lemma \ref{w4-code} implies that $|V(G')|\ge 9$ and Lemma \ref{rw4-local}\ref{w4-no-greedy-nbr} implies that $G'$ is $3$-edge-connected. It also follows from Lemma \ref{rw4-local}\ref{w4-4ecut} that $G'$ is internally $4$-edge-connected. So by minimality of $G$, Theorem \ref{rw4-thm} holds for $G'$. Since $G \succ_r G'$, $G'$ does not immerse rooted $W_4$ and thus $G'$ has one of the types 1, 2, 3, or 4. If $G'$ is type $1$, there exists $X\subset V$  with $|X|, |V\setminus X| \ge 4$ such that $d_{G'} (X) = 4$, and this contradicts Lemma \ref{rw4-local}\ref{w4-4ecut} or \ref{w4-no-deep-5ecut}. Now note that it follows from Lemma \ref{rw4-local}\ref{w4-no-greedy-nbr} that $G'$ is sausage reduced. So $|V(G')| \ge 9$ implies that $G'$ is not type $2$ or 3 either---a contradiction.
	
	So, $G$ is simple. Now, denote the root vertex of $G$ by $x$, and choose an edge $e$ which is not incident with $x$. Let $G'$ be the graph obtained from $G$ by deleting $e$, and suppressing any degree two vertices. So $|V(G')|\ge |V(G)| -2 \ge 7$, and $G'$ is $3$-edge-connected. It follows from Lemma \ref{rw4-local}\ref{w4-4ecut} that $G'$ is also internally $4$-edge-connected. So by minimality of $G$, Theorem \ref{rw4-thm} holds for $G'$. As in the last paragraph, $G'$ cannot have type $1$, otherwise $G$ would contradict either Lemma \ref{rw4-local}\ref{w4-4ecut} or \ref{w4-no-deep-5ecut}. Finally, note that $G'$ is sausage reduced. It is because if $G'[X']$ is a chain of sausages of order $\ge 3$ in $G'$, then Lemma \ref{rw4-local}\ref{w4-4ecut} implies that $e \in \delta_G (X)$, where $X$ is the set in $G$ which corresponds to $X'$ in $G'$. But then $ G[X]$ would contain parallel edges, contradicting $G$ being simple. This rules out the possibility of $G'$ having any type other than 2A. Moreover, it implies that if $G'$ has type $2$A relative to $W$, for some $W\subset V$ with $|W|=2$, then both chain of sausages between $\{x\}, W$ are of order exactly two. A quick check shows that then $G$ is not simple. This contradiction establishes Theorem \ref{rw4-thm}.
\end{proof}
\section{Forbidding a $W_4$ immersion}
\label{sec-w4}
The main goal of this section is to give the precise structure of the graphs excluding an immersion of $W_4$. As mentioned in Section \ref{sec-intro}, this problem has been studied before. Our structural theorem on the structure of graphs without a $W_4$ immersion is as follows:
\begin{theorem}
	\label{w4-thm}
	Let $G$ be $3$-edge-connected, and internally $4$-edge-connected. If $|V(G)|\ge 5$, then $G\nsucc W_4$ if and only if
	\begin{enumerate}
		\item
		\label{w4-excep-0}
		$G$ is cubic, or
		\item
		\label{w4-excep-1}
		$G$ is isomorphic to the leftmost graph in Figure \ref{fig:w4-unrooted-w-s}.
		\begin{figure}[htbp]
			\centering
			\includegraphics{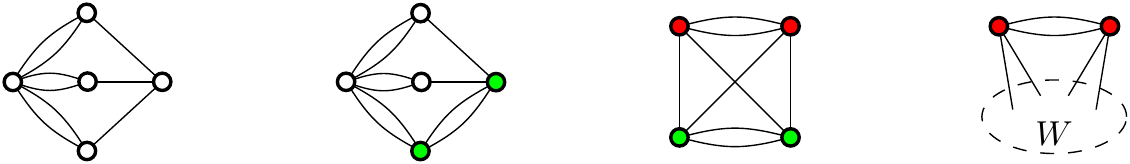}
			\caption{Non-cubic sausage reduced graphs without a $W_4$ immersion}
			\label{fig:w4-unrooted-w-s}
		\end{figure}
		\item
		\label{w4-excep-2}
		$G$ can be obtained from the second left graph above by replacing the pair of green vertices with an arbitrary chain of sausages of order at least two.
		\item
		\label{w4-excep-3}
		$G$ can be constructed from the second right graph above by replacing each pair of same-colored vertices with an arbitrary chain of sausages of order at least three.
		\item
		\label{w4-excep-4}
		There exists $W \subset V(G)$ with $1 \le |W| \le 2$ so that the graph obtained from $G$ by identifying $W$ to a single vertex is a doubled cycle, as in the rightmost structure above.
		
	\end{enumerate}
\end{theorem}
\subsection{Proof of the main theorem}
We will use our result on rooted immersions of $W_4$ to prove our theorem on unrooted immersions of $W_4$.  In addition we require the following key result.

\begin{theorem}[Mader]
\label{maderthm}
Let $G$ be a $2$-edge-connected graph and let $v \in V(G)$ satisfy $\mathrm{deg}(v) \ge 4$.  Then there exist $e,f$ incident with $v$ so that the graph $G'$ obtained from $G$ by splitting off $e,f$ satisfies $\lambda_{G'}(u,w) = \lambda_{G}(u,w)$ for every $u,w \in V \setminus \{v\}$.
\end{theorem}


\begin{proof}
	[Proof of Theorem \ref{w4-thm}]
	Let $G = (V,E)$ be a counterexample to the theorem for which $|V| + |E|$ is minimum.  We will establish a sequence of properties of $G$ eventually proving it cannot exist.
	
	\begin{enumerate}[label=(\arabic*), labelindent=0em ,labelwidth=0cm, parsep=6pt, leftmargin =7mm]
		
		\item There is a vertex of even degree in $G$. \label{geteven}
		
		If $G$ is cubic the theorem holds, so we may choose $u \in V$ with $d(u) \ge 4$.  If we treat $u$ as a root vertex, there cannot be a rooted immersion of $W_4$, so by Theorem \ref{rw4-thm} this rooted graph must have type $1, 2, 3$, or $4$.  All graphs of types 2, 3, and 4 have a vertex of even degree, so we are done unless our rooted graph has type $1$ relative to some $(U, W)$.  If $|V| = 5$, then $G$ has a vertex of even degree by parity, and else, any vertex in $V\setminus (U\cup W)$ has even degree (by part \ref{no-ht-even} of Observation \ref{seg-obs}).
		

		\item $G$ has at most two vertices of odd degree.  Furthermore, if $u \in V$ has even degree, then: \label{evenres}
	\begin{itemize}
	\item There exists $v \in N(u)$ with even degree so that $e(u,v) \ge 2$.
	\item There does not exist $w \in N(u)$ with odd degree so that $2 e(u,w) > d(w)$.
	\end{itemize}

		Let $u \in V$ have even degree (note that by \ref{geteven} such a vertex exists). 
		Now treat $u$ as a root vertex of $G$ and apply Theorem \ref{rw4-thm}.  Since $G$ does not have an immersion of $W_4$ (and $d(u)$ is even), this rooted graph must have type 1, 3, or 4.  For types 3 or 4 a straightforward check shows that $G$ has an immersion of $W_4$ unless either case \ref{w4-excep-1} or \ref{w4-excep-2} occur.  Therefore, $G$ must have a $(2,3)$-segmentation of width four relative to $(U,W)$ where $u \in U$.  By possibly removing the first or last set in this segmentation, we may further assume $|U| = 2$ and $|W|=3$.  Let $U = \{u,v\}$ and note that $4 = d( \{u,v\} ) = d(u) + d(v) - 2e(u,v)$ and it follows that $v$ has even degree and $e(u,v) \ge 2$.  Every vertex of odd degree must be contained in $W$ (by Observation \ref{seg-obs}) and $|W| \le 3$ so $G$ has at most two vertices of odd degree.  Moreover, if $w \in W$ has odd degree, then $2 e(u,w) > d(w)$ is impossible since this would give the contradiction $d(W \setminus \{w\}) \le 3$.  This completes the proof.

		\item $|V| = 5$ and $G$ has maximum degree four. \label{5limv}
		
		Suppose (for a contradiction) that \ref{5limv} is false, and (using \ref{geteven}) choose $v \in V$ so that either $d(v) \ge 5$, or both $d(v) = 4$ and $|V| > 5$. Apply Theorem \ref{maderthm} to choose a pair of edges $e,f$ incident with $v$ and let $G'$ be the graph obtained from $G$ by splitting off $e,f$ and suppressing any resulting degree 2 vertex.  It follows immediately from this construction that $G'$ is 3-edge-connected.  Suppose (for a contradiction) that $G'$ is not internally 4-edge-connected.  In this case there must exist $X \subseteq V$ with $2 \le |X| \le |V| - 2$ for which $d_{G'}(X) = 3$.  Note that by parity and \ref{evenres} both $X$ and $V \setminus X$ must contain exactly one vertex of odd degree.  There cannot exist a vertex other than $v$ with degree at least 4 in $X$ and another such vertex in $V \setminus X$ as this would contradict Theorem \ref{maderthm}.  So, by possibly switching $X$ and $V \setminus X$ we may assume that $X = \{v, u \}$ where $d_G(u) = 3$ and $d_G(v)$ is an even number at least 6.  Now \ref{evenres} implies $e_G(u,v) \le 1$ and we have a contradiction to $d_{G'}( \{v,u\}) = 3$.  We conclude that $G'$ is internally 4-edge-connected.  By construction $|V(G')| \ge 5$ so by the minimality of our counterexample, $G'$ either has an immersion of $W_4$ or cases \ref{w4-excep-0}, \ref{w4-excep-1}, \ref{w4-excep-2}, \ref{w4-excep-3}, or \ref{w4-excep-4} occur.  Since $G \succ G'$ the first of these outcomes is impossible.  Our degree assumptions imply that cases \ref{w4-excep-1} or \ref{w4-excep-2} do not happen.  A straightforward check shows that in the remaining cases the original graph $G$ either contains a $W_4$ immersion or cases \ref{w4-excep-2}, \ref{w4-excep-3} or \ref{w4-excep-4} occur.  
		
\end{enumerate}

At this point \ref{evenres} and \ref{5limv} imply that the five vertices of $G$ either all have degree 4, or three have degree 4 and two have degree 3.  Furthermore, \ref{evenres} implies that every degree 4 vertex is joined by exactly two edges to another degree 4 vertex.  It follows that $G$ has type 5 and this completes the proof.
\end{proof}

\subsection{Tree-width of graphs with no $W_4$ immersion}
In this short subsection, we will introduce the tree-width of a graph and find this parameter for  obstructions to the existence of $W_4$ immersion. As a corollary, we obtain a significant strengthening of the previously known result about immersion of $W_4$ (Theorem \ref{belmonte-w4-intro}).

A {\it tree-decomposition} of a graph $G$ consists of a tree $T$ and a family $\{W_t\}_{t \in V(T)}$ of subsets of $V(G)$ that satisfy: 
\begin{itemize}
	\item
	$V(G) = \bigcup (W_t: t \in V(T))$
	\item 
	for every edge $uv$ of $G$, there exists $t \in V(T)$ with $u, v \in W_t$
	\item for every $v \in V(G)$, the subgraph of $T$ induced by $\{ t \in V(T) \mid v \in W_t \}$ is connected.  
\end{itemize}
The parameter $\max \{|W_t|-1: t \in V(T)\}$ is called the {\it width} of the tree-decomposition. 
We say a graph $G$ has {\it tree-width} 
 $k$, and write $tw(G) =k$, if $k$ is the minimum such that $G$ has a tree-decomposition of width $k$. 
 We require one easy observation concerning tree-width.

\begin{observation}
If $H$ is a subdivision of $G$, then $tw(H) = tw(G)$.
\end{observation}

\begin{proof}
A graph with at least one edge has tree-width 1 if and only if it is a forest, so we may assume $tw(G) \ge 2$.  Since tree-width is minor monotone, $tw(G) \le tw(H)$.  To prove the other inequality, it suffices to consider the case when $H$ is obtained from $G$ by subdividing the edge $uv$ with a new vertex $w$.  Consider an optimal tree decomposition of $G$ and choose a vertex $t$ of the tree with $u,v \in W_t$.  Now adding a leaf vertex $t'$ to the tree adjacent only to $t$ and setting $W_{t'} = \{u,v,w\}$ gives a tree decomposition of $H$ of the same width.  
\end{proof}

It follows from the above observation that if $H$ is a graph obtained from $G$ by sausage reducing it, then $tw (G) = tw (H)$. Now note that $K_{2, 3}$ and any graph on at most four vertices has tree-width at most three. This gives the following corollary of Theorem \ref{w4-thm}.

\begin{corollary}
Let $G$ be a $3$-edge-connected and internally $4$-edge-connected graph which does not immerse $W_4$. Then $G$ either is cubic or has tree-width at most $3$.
\end{corollary}

\bibliographystyle{plain}
\bibliography{references}
\end{document}